\UseRawInputEncoding %arxiv
\documentclass[11pt,a4paper]{article}
\usepackage{amsfonts}
\usepackage{amssymb}
\usepackage{mathrsfs}
\usepackage{amsmath}
\usepackage{booktabs}
\usepackage{epsf,epsfig,amsfonts,amsgen,indentfirst}
\usepackage{amsmath,amstext,amsbsy,amsopn,amsthm,bbding,wasysym}
\usepackage{multicol,mathdots}
\usepackage{subfigure}
\usepackage[numbers,sort&compress]{natbib}
\allowdisplaybreaks

\newtheorem{theorem}{Theorem}[section]

\newtheorem{lemma}[theorem]{Lemma}

\newtheorem{definition}{Definition}[section]

\newtheorem{conjecture}[theorem]{Conjecture}

\begin{document}
\title
{\LARGE \textbf{On the Tutte polynomial of series-parallel posets
\thanks{Supported by NSFC (Nos. 12261071,  12571019) and NSF of Qinghai Province (No. 2025-ZJ-902T).} }}

\author{Jianxuan Luo$^a$, Tingzeng Wu$^{b,c}$\thanks{{Corresponding author.\newline
\emph{E-mail addresses}: mathtzwu@163.com, jianxuanluo@163.com, hjlai2015@hotmail.com
}}, Hong-Jian Lai$^{d}$ \\
{\small $^{a}$ School of Mathematical Sciences, Xiamen University,}\\
{\small  Xiamen, Fujian 361005,  P.R.~China} \\
{\small $^{b}$ School of Mathematics and Statistics, Qinghai Minzu University, }\\
{\small  Xining, Qinghai 810007, P.R.~China} \\
{\small $^{c}$ Qinghai Institute of Applied Mathematics, Xining, Qinghai 810007, P.R.~China}\\
{\small $^{d}$ School of Mathematics and System Sciences, Guangdong Polytechnic Normal University,}\\
{\small  Guangzhou, Guangdong 510665, P.R.~China}}
\date{}

\maketitle
\noindent {\bf Abstract:}
The Tutte polynomial is a bivariate polynomial that has been extensively studied in graph and matroid theory.
%Gordon first studied the Tutte polynomial of the greedoid induced by a poset $P$, denoted by $T(P;x,y)$.
%Let $P$ be a poset, and let $T(P;x,y)$ denote its Tutte polynomial.
%A poset is called series-parallel if it can be constructed recursively from the one-element poset $\mathbf{1}$ using direct sums and ordinal sums.
Gordon was the first to study the Tutte polynomial $T(P;x,y)$ of the greedoid induced by a poset $P$, 
including the special case of series-parallel posets.
In particular, Gordon and McMahon conjectured that, for any two series-parallel posets $P$ and $Q$, the equality $T(P;x,y)=T(Q;x,y)$ holds if and only if $P\cong Q$.
In studying this conjecture, Gordon introduced a subclass $\mathcal{P}$ of series-parallel posets and proved that this equivalence holds for all $P,Q\in\mathcal{P}$.
%In this paper, we investigate this conjecture by introducing a new subclass $\mathrm H$ of series-parallel posets 
%and proving that $\mathcal{P}\subseteq\mathrm H$. We extend Gordon's result by allowing $P$ to belong to $\mathrm H$ and $Q$ to be an arbitrary series-parallel poset,
%Using unique factorization of polynomials together with the decompositions of posets into direct sums and ordinal sums,
%and prove that, for every $P\in\mathrm H$ and every series-parallel poset $Q$, the equality $T(P;x,y)=T(Q;x,y)$ holds if and only if $P\cong Q$.
%Consequently, Gordon's result for the subclass $\mathcal{P}$ is a special case of our theorem.
In this paper, we introduce a new subclass $\mathrm H$ of series-parallel posets and prove that $\mathcal P\subseteq\mathrm H$.
Moreover, we show that, for every $P\in\mathrm H$ and every $Q\in\mathrm{SP}$, the equality $T(P;x,y)=T(Q;x,y)$ holds if and only if $P\cong Q$,
thereby extending Gordon's result.

%\smallskip

\section{Introduction}

Unless otherwise stated, all graphs are finite, simple, and undirected.
All posets considered in this paper are finite.
For a graph $G$, let $V(G)$ and $E(G)$ denote its vertex set and edge set, respectively, 
and let $\overline{G}$ denote its complement.
A partially ordered set (poset) $P$ is a set equipped with a binary relation $\leq$ that is reflexive, antisymmetric, and transitive.
We also use $P$ to denote its ground set and write $|P|$ for its size.
For any $u,v\in P$, we write $u<v$ if $u\leq v$ and $u\neq v$.
Two elements $u,v\in P$ are called comparable if $u\leq v$ or $v\leq u$.
Otherwise, they are called incomparable.
A subset of $P$ is called a chain if its elements are pairwise comparable,
and an antichain if any two distinct elements in it are incomparable.
Let $\mathbf{1}$ denote the one-element poset.
Two posets $P$ and $Q$ are isomorphic as posets, written
$P\cong Q$, if there exists a bijection $\varphi:P\to Q$ such that,
for all $u,v\in P$, we have $u\leq v$ in $P$ if and only if
$\varphi(u)\leq\varphi(v)$ in $Q$.
For a set $X$, let $2^X$ denote its power set.
We use $y$ as an additional indeterminate, and $\deg_x F$
denotes the degree of a nonzero polynomial $F$ in $x$.
Furthermore, let $\mathbb{Z}$ and $\mathbb{Q}$ denote the ring of integers and the field of rational numbers, respectively, 
and let $\mathbb{Q}[x,y]$ denote the polynomial ring in the indeterminates $x$ and $y$ over $\mathbb{Q}$.

The Tutte polynomial is a bivariate polynomial that has been extensively studied in graph and matroid theory
\cite{BauerEtAl2026, BrylawskiOxley1992, ChaudharyGordon1991, ColbournProvanVertigan1995}.
Gordon and McMahon \cite{GordonMcMahon1989} extended the Tutte polynomial to greedoids:

\begin{definition}\label{def:greedoid}
A greedoid $\Gamma$ is an ordered pair $(E,\mathcal F)$, where $E$ is a finite set 
and $\mathcal F$ is a nonempty family of subsets of $E$ satisfying the following axioms:
\begin{enumerate}
 \item[\textup{(i)}] $\varnothing\in\mathcal F$;
 \item[\textup{(ii)}] For any $F,F'\in\mathcal F$, if $|F'|<|F|$, 
 then there exists $x\in F\setminus F'$ such that $F'\cup\{x\}\in\mathcal F$.
\end{enumerate}
\end{definition}

The set $E$ is called the ground set of $\Gamma$, and the members of
$\mathcal F$ are called the feasible sets of $\Gamma$.
The rank of a subset $A\subseteq E$ is defined by
\begin{eqnarray*}
\rho_{\Gamma}(A)=\max\{|A'|:A'\subseteq A,\ A'\in\mathcal F\},
\end{eqnarray*}
and we set $\rho(\Gamma)=\rho_{\Gamma}(E)$.
% Generalizing the definition of the Tutte polynomial of a matroid,
% McMahon and Gordon defined the Tutte polynomial of a greedoid
% in \cite{GordonMcMahon}.
The Tutte polynomial of a greedoid $\Gamma$ with ground set $E$
and rank function $\rho_\Gamma$ is defined by
\begin{eqnarray*}
T(\Gamma;x,y)
= \sum_{A\subseteq E}
(x-1)^{\rho(\Gamma)-\rho_\Gamma(A)}
(y-1)^{|A|-\rho_\Gamma(A)}.
\end{eqnarray*}
For related work on matroids and on the Tutte polynomials of matroids and greedoids, see
\cite{Bjorner1985, GeelenZhou2006, GordonMcMahon2022, KnappNoble2025, LuoWuLai2026, McMahon1993, Zhou2004, Zhou2018}.
In particular, Gordon \cite{Gordon1993} studied the Tutte polynomial of the greedoid
associated with the order ideals of a poset.
For a poset $P$, a subset $I\subseteq P$ is an order ideal if,
whenever $u\in I$ and $v\in P$ satisfy $v\leq u$, we have $v\in I$.
Let $\mathcal J(P)$ be the family of all order ideals of $P$.
The pair $\Gamma_P=(P,\mathcal J(P))$ is a greedoid.
A greedoid whose feasible sets are closed under unions is
called an antimatroid. Thus, $\Gamma_P$ is the poset antimatroid associated with $P$. 
For $S\subseteq P$, its rank is
\begin{eqnarray*}
\rho_P(S)=\rho_{\Gamma_P}(S)=\max\{|I|:I\subseteq S,\ I\in\mathcal J(P)\}.
\end{eqnarray*}
Since $P$ is itself an order ideal, $\rho(\Gamma_P)=|P|$. We define
\begin{eqnarray*}
T(P;x,y)=T(\Gamma_P;x,y)=\sum_{S\subseteq P}(x-1)^{|P|-\rho_P(S)}(y-1)^{|S|-\rho_P(S)}.
\end{eqnarray*}

Gordon \cite{Gordon1993} focused in particular on series-parallel posets
and their various subclasses.
We now introduce the definition of a series-parallel poset:

\begin{definition}\label{def:poset-sums-dual}
Let $P$ and $Q$ be two disjoint posets.
\begin{enumerate}
\item The direct sum of $P$ and $Q$, denoted by $P+Q$, is the poset
on $P\cup Q$ in which $x\le y$ if either
\begin{enumerate}
    \item[\textup{(i)}] $x,y\in P$ and $x\le y$ in $P$ or
    \item[\textup{(ii)}] $x,y\in Q$ and $x\le y$ in $Q$.
\end{enumerate}

\item The ordinal sum of $P$ and $Q$, denoted by $P\oplus Q$, is the poset
on $P\cup Q$ in which $x\le y$ if either
\begin{enumerate}
    \item[\textup{(i)}] $x,y\in P$ and $x\le y$ in $P$ or
    \item[\textup{(ii)}] $x,y\in Q$ and $x\le y$ in $Q$ or
    \item[\textup{(iii)}] $x\in P$ and $y\in Q$.
\end{enumerate}

\item The dual poset $P^*$ has the same ground set as $P$, with
$x\le y$ in $P^*$ if and only if $y\le x$ in $P$.
\end{enumerate}
\end{definition}

%We let $\mathbf{1}$ be the one-element poset.
%A poset is a series-parallel poset if it can be built up recursively from $\mathbf{1}$ 
%by using the operations of direct sum and ordinal sum.
%The Hasse diagrams of the posets formed by the above operations
%can all be obtained from the Hasse diagrams of $P$ and $Q$ by straightforward techniques.
%The reader can consult \cite{Stanley1986} or work out the details directly.
%Finally, the dual $P^*$ of a poset $P$ is obtained by flipping the Hasse diagram of $P$, i.e., $x\leq y$ in $P^*$
%iff $y\leq x$ in $P$.

A nonempty poset is called series-parallel if it can be constructed,
up to isomorphism, from $\mathbf1$ by repeated applications of direct sums
and ordinal sums.
We denote the class of all series-parallel posets by $\mathrm{SP}$.

Series-parallel posets are of interest because their recursive structure
allows the development of many algorithms that run in polynomial time.
The recursive structure of series-parallel posets supports
efficient algorithms for several computational problems.
Valdes et al. \cite{Valdes1981} gave a linear-time recognition
algorithm for the corresponding series-parallel digraphs.
%Gordon \cite{Gordon1996} proved that the Tutte polynomial
%of a series-parallel poset can be computed in polynomial time.
Many other authors have also studied $\mathrm{SP}$ posets
for various purposes.
For example, Stanley \cite{Stanley1974} used P\'olya's theorem
to obtain a generating function for the number of $\mathrm{SP}$ posets
with $n$ elements.
On the other hand, let $N(P)$ denote the number of order ideals of $P$,
and let $N(x)$ denote the number of order ideals of $P$ that contain
the element $x$.
Faigle et al. \cite{Faigle1986} proved that an element $x$
satisfying $\frac14\leq N(x)/N(P)\leq\frac34$ can be found
efficiently in an $\mathrm{SP}$ poset, and that these bounds
are best possible. In contrast, Provan and Ball \cite{Provan1983} proved that
computing $N(P)$ for a general poset is a $\#\mathrm P$-complete problem.

Gordon and McMahon \cite{GordonMcMahon2022} considered the Tutte
polynomials of greedoids induced by series-parallel posets
and proposed the following conjecture:

\begin{conjecture}(\cite{GordonMcMahon2022})\label{conj:reconstruction}
Let $P$ and $Q$ be any two series-parallel posets. Then
\begin{eqnarray*}\label{eq:conj}
T(P;x,y)=T(Q;x,y)\quad\text{if and only if}\quad P\cong Q.
\end{eqnarray*}
\end{conjecture}

Gordon \cite{Gordon1996} obtained a partial result on
Conjecture \ref{conj:reconstruction} by proving it for a subclass
$\mathcal{P}$ of series-parallel posets.
To define this subclass, let
\begin{eqnarray*}
\mathbf{n}=\underbrace{\mathbf{1}+\mathbf{1}+\cdots+\mathbf{1}}_{n}.
\end{eqnarray*}
The subclass $\mathcal{P}$ is defined recursively by the following rules:
\begin{enumerate}
    \item[\textup{(i)}] The one-element poset $\mathbf{1}$
    belongs to $\mathcal{P}$.
    \item[\textup{(ii)}] If $P\in\mathcal{P}$, then
    $P\oplus\mathbf{n}\in\mathcal{P}$ for every integer $n\geq 1$.
    \item[\textup{(iii)}] If $P,Q\in\mathcal{P}$, then
    $P+Q\in\mathcal{P}$.
    \item[\textup{(iv)}] If $P\in\mathcal{P}$, then
    $P^*\in\mathcal{P}$.
\end{enumerate}

\begin{theorem}(\cite{Gordon1996})\label{thm:theorem33-43}
Let $P,Q\in\mathcal{P}$. Then
\begin{eqnarray*}
T(P;x,y)=T(Q;x,y)\quad\text{if and only if}\quad P\cong Q.
\end{eqnarray*}
\end{theorem}

This conjecture remains open.
In this paper, we establish a partial affirmative answer to Conjecture \ref{conj:reconstruction}.
We first introduce the notation needed to define the class $\mathrm H$.

Let $P\in\mathrm{SP}$.
For any $W\subseteq P$, let $P[W]$ denote the induced subposet of $P$
on the ground set $W$, whose order relation is defined as follows:
for any $s,t\in W$, we have $s\le t$ in $P[W]$ if and only if $s\le t$ in $P$.
Let $G=G(P)$ be the comparability graph of $P$, that is,
$V(G)=P$ and $E(G)=\{\{s,t\}:s,t\in P,\ s<t\}$.
Similarly, let $\overline G$ be the incomparability graph of $P$,
that is, $V(\overline G)=P$ and $E(\overline G)=\{\{s,t\}:s,t\in P,\ s\neq t,\ s\nleq t,\ t\nleq s\}$.
Thus, two distinct vertices are adjacent in $\overline G$ if and only if they are incomparable in $P$,
so $\overline G$ is precisely the complement of $G$.

For any $W\subseteq P$, let $G[W]$ and $\overline G[W]$ denote 
the subgraphs of $G$ and $\overline G$, respectively, induced by $W$.
For any graph $G$, let $\mathcal C(G)$ denote the set of all connected components of $G$.
Let $\mathcal W(P)\subseteq 2^P\setminus\{\varnothing\}$
be the smallest family with respect to inclusion such that
$P\in\mathcal W(P)$ and, for every $W\in\mathcal W(P)$ and every
$C\in\mathcal C(G[W])\cup\mathcal C(\overline G[W])$, we have $V(C)\in\mathcal W(P)$.

\begin{definition}\label{def:poset-h}
Define $\mathrm H$ to be the class of all posets $P\in\mathrm{SP}$
satisfying the following condition:
for every $W\in\mathcal W(P)$, if $G(P[W])$ is disconnected,
then, for every $C\in\mathcal C(G(P[W]))$,
the polynomial $T(P[V(C)];x,y)$ is irreducible in $\mathbb Q[x,y]$,
where $\mathbb Q$ denotes the field of rational numbers.
\end{definition}

The inclusion $\mathcal P\subseteq\mathrm H$ is not immediate and we prove the following theorem:

\begin{theorem}\label{thm:main-sub}
Let $\mathcal P$ and $\mathrm H$ be defined as above.
Then $\mathcal P\subseteq\mathrm H$.
\end{theorem}

%Next, we give the main theorem. In fact,
%combining the above inclusion with our main theorem \ref{thm:main}, we obtain Gordon's result
%\cite{Gordon1996}, stated in Theorem \ref{thm:theorem33-43}, as a special case of Theorem \ref{thm:main}.
Our main result is the following theorem. Together with Theorem \ref{thm:main-sub}, it recovers
Theorem \ref{thm:theorem33-43} as a special case.

\begin{theorem}\label{thm:main}
Let $P\in\mathrm H$ and $Q\in\mathrm{SP}$ be arbitrary. Then
\begin{eqnarray*}
T(P;x,y)=T(Q;x,y)\quad\text{if and only if}\quad P\cong Q.
\end{eqnarray*}
\end{theorem}

%\begin{remark}
%The inclusion $\mathcal P\subseteq\mathrm H$ is not immediate and is proved in Theorem \ref{thm:main-sub}.
%Combining this inclusion with our main theorem \ref{thm:main}, we obtain Gordon's result
%\cite{Gordon1996}, stated in Theorem \ref{thm:theorem33-43}, as a special case of Theorem \ref{thm:main}.
%\end{remark}

This paper focuses on Conjecture \ref{conj:reconstruction}.
%In Section 2, we introduce several lemmas needed for the proof of the main theorem.
%In Section 3, we establish the inclusion $\mathcal P\subseteq\mathrm H$.
%Furthermore, we prove that, for every $P\in\mathrm H$ nd every $Q\in\mathrm{SP}$, the equality
%$T(P;x,y)=T(Q;x,y)$ holds if and only if $P\cong Q$.
%Furthermore, we establish the inclusion $\mathcal P\subseteq\mathrm H$,
%showing that the subclass $\mathrm H$ of series-parallel posets
%constructed in this paper contains $\mathcal P$.
%We thus extend Gordon's result \cite{Gordon1996}, stated in Theorem \ref{thm:theorem33-43},
%by replacing the assumptions $P,Q\in\mathcal P$ with $P\in\mathrm H$ and $Q\in\mathrm{SP}$.
%Consequently, Gordon's result for the subclass $\mathcal P$ is recovered as a special case of our main theorem.
The remainder of this paper is organized as follows.
Section 2 presents the lemmas needed for the main results.
Section 3 proves Theorems \ref{thm:main-sub} and \ref{thm:main}.
Section 4 concludes the paper.

\section{Some lemmas}

In this section, we retain the notation introduced in Section 1
and present several lemmas needed for the proof of the main theorem.

For a graph $G$, let $c(G)=|\mathcal C(G)|$ denote its number of connected components.
A poset $P$ is called connected if $c(G(P))=1$.
By the definition of an induced subposet, for every nonempty $W\subseteq P$, 
we have $G(P[W])=G(P)[W]$ and $\overline{G(P[W])}=\overline{G(P)}[W]$.
Let $G$ and $H$ be undirected graphs with disjoint vertex sets,
and denote their disjoint union by $G\cup H$.
The join of $G$ and $H$, denoted by $G\vee H$, has vertex set $V(G)\cup V(H)$ and edge set
\begin{eqnarray*}
E(G\vee H)
=E(G)\cup E(H)\cup
\bigl\{\{u,v\}:u\in V(G),\ v\in V(H)\bigr\}.
\end{eqnarray*}
By the definitions of the direct sum and the ordinal sum,
for posets $P$ and $Q$ with disjoint ground sets, we have
\begin{eqnarray}\label{eq:graph-sums}
G(P+Q)=G(P)\cup G(Q),\qquad
G(P\oplus Q)=G(P)\vee G(Q).
\end{eqnarray}
Let $\vec G(P)$ denote the digraph determined by the strict order relation of $P$,
with vertex set $V(\vec G(P))=P$ and arc set $A(\vec G(P))=\{(u,v)\in P\times P:u<v\}$.
For any $W\subseteq P$, let $\vec G(P)[W]$ denote the induced subdigraph of $\vec G(P)$ on $W$,
whose arc set is $A(\vec G(P))\cap(W\times W)$.
By the definition of an induced subposet, we have $\vec G(P[W])=\vec G(P)[W]$.
A bijection $\varphi:P\to Q$ is a poset isomorphism if and only if, 
for all distinct $u,v\in P$, we have
\begin{eqnarray*}
(u,v)\in A(\vec G(P))
\quad\text{if and only if}\quad
(\varphi(u),\varphi(v))\in A(\vec G(Q)).
\end{eqnarray*}
Consequently, $P\cong Q$ if and only if the digraphs $\vec G(P)$ and $\vec G(Q)$ are isomorphic.

A nonempty poset $P$ is called ordinally indecomposable 
if there do not exist two nonempty posets $R$ and $S$ such that $P\cong R\oplus S$.
Let $P$ be a nonempty poset.
The connected components of the comparability graph $G(P)$ 
determine a decomposition of $P$ as a direct sum of nonempty factors
that are indecomposable with respect to direct sums.
This decomposition is unique up to the order of its factors.
The connected components of the incomparability graph $\overline{G(P)}$, 
arranged in the order determined by the partial order on $P$, 
determine the unique decomposition of $P$ as an ordinal sum of nonempty ordinally indecomposable factors.
When $\overline{G(P)}$ is connected, this ordinal sum decomposition consists of the single factor $P$.

\begin{lemma}\label{lem:graph-decomposition}
Let $P\in\mathrm{SP}$. Then the following statements hold:
\begin{enumerate}
    \item[\textup{(i)}] For every nonempty $W\subseteq P$, we have $P[W]\in\mathrm{SP}$.
    \item[\textup{(ii)}] Suppose that $\mathcal C(G(P))=\{C_1,\ldots,C_s\}$,
    and let $P_i=P[V(C_i)]$ for $1\le i\le s$.
    Then $P=P_1+\cdots+P_s$.
    This decomposition into nonempty factors that are indecomposable
    with respect to direct sums is unique up to the order of its factors.
    \item[\textup{(iii)}]
    If $|P|\ge2$, then exactly one of $G(P)$ and $\overline{G(P)}$ is disconnected.
    Under this assumption, if $c(G(P))=1$, then the connected
    components of $\overline{G(P)}$ admit a unique ordering
    $D_1,\ldots,D_t$ such that, for every $1\le i<j\le t$,
    we have $V(D_i)\times V(D_j)\subseteq A(\vec G(P))$.
    In this case, $t\ge2$ and $P=P[V(D_1)]\oplus\cdots\oplus P[V(D_t)]$,
    where each factor is ordinally indecomposable.
\end{enumerate}
\end{lemma}

\begin{proof}
\textup{(i)}
Let $N$ be the four-element poset on the ground set $\{u,v,w,z\}$
whose only strict order relations are $u<v$, $w<v$, and $w<z$.
By the characterization in \cite{Valdes1981},
a finite nonempty poset belongs to $\mathrm{SP}$ if and only if
it contains no induced subposet isomorphic to $N$.
Let $W\subseteq P$ be an arbitrary nonempty subset.
Suppose that $P[W]\notin\mathrm{SP}$.
By the above characterization, there exists $U\subseteq W$
such that $(P[W])[U]\cong N$.
On the other hand, by the definition of an induced subposet,
for any $s,t\in U$, we have $s\le t$ in $(P[W])[U]$
if and only if $s\le t$ in $P[W]$, which in turn holds
if and only if $s\le t$ in $P$.
Thus, $(P[W])[U]$ and $P[U]$ have the same ground set
and the same order relation, and hence $(P[W])[U]=P[U]$.
It follows that $P[U]\cong N$, so $P$ contains an induced subposet isomorphic to $N$.
This contradicts $P\in\mathrm{SP}$ and the above characterization.
Therefore, $P[W]\in\mathrm{SP}$.
This completes the proof of \textup{(i)}.

\textup{(ii)}
Since $\mathcal C(G(P))=\{C_1,\ldots,C_s\}$,
the vertex sets of these connected components are nonempty,
pairwise disjoint, and satisfy $P=\bigcup_{i=1}^{s}V(C_i)$.
Let $P_i=P[V(C_i)]$.
By the definitions of an induced subposet and the comparability graph,
we have $G(P_i)=G(P)[V(C_i)]=C_i$.
Thus, each $P_i$ is a nonempty connected poset.
Let $i\neq j$, and take arbitrary elements $u\in V(C_i)$ and $v\in V(C_j)$.
If $u$ and $v$ are comparable in $P$, then $e=\{u,v\}\in E(G(P))$.
Hence, $u$ and $v$ belong to the same connected component, contrary to $C_i\neq C_j$.
Therefore, elements belonging to distinct $P_i$ are incomparable in $P$.
Moreover, the order relations between elements of each $P_i$ are the same as those in $P$.
It follows from the definition of the direct sum 
in Definition \ref{def:poset-sums-dual} that $P=P_1+\cdots+P_s$.

We now prove uniqueness.
Let $P=R_1+\cdots+R_t$ be another such decomposition,
where each $R_j$ is a nonempty connected poset.
By the definition of the direct sum in Definition \ref{def:poset-sums-dual}
and the definition of the comparability graph, we obtain
\begin{eqnarray*}
G(P)=G(R_1)\cup\cdots\cup G(R_t).
\end{eqnarray*}
Since each $G(R_j)$ is connected and no edge of $G(P)$ joins vertices belonging to distinct $G(R_j)$,
the graphs $G(R_1),\ldots,G(R_t)$ are precisely the connected components of $G(P)$.
By the uniqueness of the decomposition of a graph into connected components, we have $t=s$.
After reindexing the direct summands $R_1,\ldots,R_s$,
we may assume that $G(R_j)=C_j$ for all $1\leq j\leq s$.
Thus, the ground set of $R_j$ is $V(C_j)$.
Furthermore, since $P=R_1+\cdots+R_s$,
the order relations between elements of each $R_j$ are the same as those in $P$.
Hence, $R_j=P[V(C_j)]=P_j$ for all $1\leq j\leq s$.
Therefore, this direct sum decomposition is unique up to the order of its factors.
This completes the proof of \textup{(ii)}.

\textup{(iii)}
Suppose that $|P|\geq 2$.
By the recursive definition of $\mathrm{SP}$, 
there exist nonempty posets $R,S\in\mathrm{SP}$ with disjoint ground sets 
such that $P=R+S$ or $P=R\oplus S$.
If $P=R+S$, then the formula \eqref{eq:graph-sums} gives $G(P)=G(R)\cup G(S)$.
Since $R$ and $S$ are both nonempty and there are no edges between $G(R)$ and $G(S)$, 
the graph $G(P)$ is disconnected.
If $P=R\oplus S$, then the formula \eqref{eq:graph-sums} gives $G(P)=G(R)\vee G(S)$, 
and hence $\overline{G(P)}=\overline{G(R)}\cup\overline{G(S)}$.
Therefore, $\overline{G(P)}$ is disconnected.
Thus, at least one of $G(P)$ and $\overline{G(P)}$ is disconnected.

We next prove that $G(P)$ and $\overline{G(P)}$ cannot both be disconnected.
Suppose that $G(P)$ is disconnected, and let $u,v\in P$ be arbitrary distinct vertices.
If $u$ and $v$ belong to different connected components of $G(P)$,
then $\{u,v\}\notin E(G(P))$, and hence $\{u,v\}\in E(\overline{G(P)})$.
If $u$ and $v$ belong to the same connected component,
then, since $G(P)$ is disconnected, we may choose a vertex $w$ from another connected component.
In this case, $\{u,w\},\{w,v\}\notin E(G(P))$, so $\{u,w\},\{w,v\}\in E(\overline{G(P)})$.
Thus, $u,w,v$, in this order, form a path in $\overline{G(P)}$.
Consequently, any two distinct vertices of $\overline{G(P)}$
are joined by a path, and therefore $\overline{G(P)}$ is connected.
This shows that $G(P)$ and $\overline{G(P)}$ cannot both be disconnected.
Since at least one of these two graphs is disconnected,
exactly one of them is disconnected.

Now suppose that $c(G(P))=1$. Then $t=c(\overline{G(P)})\geq 2$.
For any distinct $u,v\in P$, the definitions of the comparability graph 
and the complement imply that $\{u,v\}\in E(\overline{G(P)})$ if and only if
$\{u,v\}\notin E(G(P))$, which is equivalent to $u$ and $v$ being incomparable in $P$.
Since $V(\overline{G(P)})=P$, the graph $\overline{G(P)}$ 
is precisely the incomparability graph of $P$.
By the result in \cite{ErneHeitzigReinhold2002},
every nonempty poset admits a unique decomposition
as an ordinal sum of nonempty ordinally indecomposable factors,
and the ground sets of these factors are precisely the vertex sets
of the connected components of its incomparability graph.
Therefore, the connected components of $\overline{G(P)}$
can be uniquely ordered as $D_1,\ldots,D_t$ such that
$P=P[V(D_1)]\oplus\cdots\oplus P[V(D_t)]$,
where each $P[V(D_i)]$ is nonempty and ordinally indecomposable.
By the definition of the ordinal sum, for every $1\leq i<j\leq t$, 
every $u\in V(D_i)$ and every $v\in V(D_j)$ satisfy $u<v$ in $P$.
The definition of the arc set of $\vec G(P)$ then gives $(u,v)\in A(\vec G(P))$.
Hence, 
\begin{eqnarray*}
V(D_i)\times V(D_j)\subseteq A(\vec G(P))\quad(1\leq i<j\leq t).
\end{eqnarray*}
This proves \textup{(iii)}.
\end{proof}

%There are several ways to give a poset $P$ an antimatroid structure.
%We use order ideals to define the poset antimatroid.
%A set $I$ is anorder ideal if $x\in I$ and $y< x$ implies $y\in I$.
%Dually, a subset $F$ is an order filter if $x\in F$ and $y\geq x$ implies $y\in F$.
%The feasible sets of the poset antimatroid are the order ideals of the poset.

Recall that a subset $I\subseteq P$ is an order ideal if,
for all $u\in I$ and $v\in P$, the relation $v\leq u$ implies
$v\in I$. Dually, a subset $F\subseteq P$ is an order filter if,
for all $u\in F$ and $v\in P$, the relation $u\leq v$ implies
$v\in F$. The feasible sets of $\Gamma_P$ are the order ideals
of $P$.

For an antichain $A$ of $P$ (i.e., a subset in which each pair of different elements
is incomparable), let $I(A)$ and $I^*(A)$ denote the order ideal and the order filter generated by $A$,
respectively. Thus,
\begin{eqnarray*}
I(A)=\{u\in P:u\leq v\text{ for some }v\in A\},
\end{eqnarray*}
and
\begin{eqnarray*}
I^*(A)=\{u\in P:v\leq u\text{ for some }v\in A\}.
\end{eqnarray*}

\begin{lemma}(\cite{Gordon1993})\label{lem:antichain}
Let $P$ be a finite poset, and let $\mathcal A(P)$ denote
the set of all antichains of $P$, including the empty antichain. Then
\begin{equation}\label{eq:antichain-expansion}
T(P;x,y)=\sum_{A\in\mathcal A(P)}(x-1)^{|I^*(A)|}y^{|I^*(A)|-|A|}.
\end{equation}
\end{lemma}

\begin{lemma}\label{lem:H-inheritance}
For every $P\in\mathrm{SP}$, the family $\mathcal W(P)$
exists and is unique.
For every $W\in\mathcal W(P)$, we have
\begin{equation}\label{eq:W-inheritance}
\mathcal W(P[W])\subseteq\{U\in\mathcal W(P):U\subseteq W\}.
\end{equation}
The class $\mathrm H$ is invariant under poset isomorphism.
If $P\in\mathrm H$ and $W\in\mathcal W(P)$, then
$P[W]\in\mathrm H$.
In particular, when a poset $P\in\mathrm H$ is decomposed
as a direct sum of nonempty factors that are indecomposable
with respect to direct sums, or as an ordinal sum of nonempty
ordinally indecomposable factors, every factor belongs to $\mathrm H$.
Moreover, if $P\in\mathrm H$ and $c(G(P))\geq 2$,
then the Tutte polynomial of each factor in its direct sum
decomposition into nonempty factors that are indecomposable
with respect to direct sums is irreducible in $\mathbb Q[x,y]$.
\end{lemma}

\begin{proof}
We first prove the existence and uniqueness of $\mathcal W(P)$.
Write $G=G(P)$, and let $\mathfrak F(P)$ denote the set
of all families $\mathcal K\subseteq 2^P\setminus\{\varnothing\}$
satisfying the following conditions:
$P\in\mathcal K$ and, for every $W\in\mathcal K$ and every
$C\in\mathcal C(G[W])\cup\mathcal C(\overline G[W])$, we have $V(C)\in\mathcal K$.
Since $P$ is nonempty, we have $P\in 2^P\setminus\{\varnothing\}$.
For every nonempty $W\subseteq P$, each connected component $C$
of either $G[W]$ or $\overline G[W]$ satisfies $\varnothing\neq V(C)\subseteq W\subseteq P$.
Therefore, $2^P\setminus\{\varnothing\}\in\mathfrak F(P)$,
and hence $\mathfrak F(P)\neq\varnothing$. Let
\begin{eqnarray}\label{eq:W0}
\mathcal W_0=\bigcap_{\mathcal K\in\mathfrak F(P)}\mathcal K.
\end{eqnarray}
Then $\mathcal W_0\subseteq 2^P\setminus\{\varnothing\}$.
Since every $\mathcal K\in\mathfrak F(P)$ contains $P$,
we also have $P\in\mathcal W_0$.
Take any $W\in\mathcal W_0$ and any $C\in\mathcal C(G[W])\cup\mathcal C(\overline G[W])$.
For every $\mathcal K\in\mathfrak F(P)$,
the inclusion $\mathcal W_0\subseteq\mathcal K$ implies that $W\in\mathcal K$.
By the defining conditions of $\mathcal K$, we have $V(C)\in\mathcal K$.
Since $V(C)\in\mathcal K$ for every
$\mathcal K\in\mathfrak F(P)$, the formula \eqref{eq:W0}
implies that $V(C)\in\mathcal W_0$.
Thus, $\mathcal W_0$ satisfies all the above conditions, so $\mathcal W_0\in\mathfrak F(P)$.
On the other hand, $\mathcal W_0\subseteq\mathcal K$ for every $\mathcal K\in\mathfrak F(P)$.
Hence, $\mathcal W_0$ is the smallest family with respect to inclusion satisfying the above conditions.
Therefore, $\mathcal W(P)$ exists and is unique, and $\mathcal W(P)=\mathcal W_0$.

We now prove the formula \eqref{eq:W-inheritance}.
Fix $W\in\mathcal W(P)$.
%By \textup{(i)} of Lemma \ref{lem:graph-decomposition},
%we have $P[W]\in\mathrm{SP}$.
Since $P\in\mathrm{SP}$,
part \textup{(i)} of Lemma \ref{lem:graph-decomposition}
gives $P[W]\in\mathrm{SP}$.
Let $\mathcal B=\{U\in\mathcal W(P):U\subseteq W\}$.
We verify that $\mathcal B$ satisfies the conditions required
in the definition of $\mathcal W(P[W])$.
First, $\mathcal B\subseteq 2^W\setminus\{\varnothing\}$.
Moreover, since $W\in\mathcal W(P)$, we have $W\in\mathcal B$.
Next, take any $U\in\mathcal B$.
By the definition of $\mathcal B$, we have $U\in\mathcal W(P)$ and $U\subseteq W$.
The definitions of an induced subposet and an induced subgraph give
$G(P[W])=G[W]$ and $\overline{G(P[W])}=\overline G[W]$.
Consequently, $G(P[W])[U]=(G[W])[U]=G[U],\overline{G(P[W])}[U]=(\overline G[W])[U]=\overline G[U]$.
Take any $C\in\mathcal C(G(P[W])[U])\cup\mathcal C(\overline{G(P[W])}[U])$.
By the above identities, we have $C\in\mathcal C(G[U])\cup\mathcal C(\overline G[U])$.
Since $U\in\mathcal W(P)$, the definition of $\mathcal W(P)$ implies that $V(C)\in\mathcal W(P)$.
At the same time, $V(C)\subseteq U\subseteq W$, and hence $V(C)\in\mathcal B$.

Thus, $\mathcal B$ is a family of nonempty subsets of $W$
that contains $W$ and satisfies the closure condition in the definition of $\mathcal W(P[W])$.
Since $\mathcal W(P[W])$ is the smallest such family with respect to inclusion, we obtain
\begin{eqnarray*}
\mathcal W(P[W])\subseteq\mathcal B=\{U\in\mathcal W(P):U\subseteq W\}.
\end{eqnarray*}

We next prove that $\mathrm H$ is invariant under poset isomorphism.
Let $P\in\mathrm{SP}$, let $\varphi:P\to Q$ be a poset isomorphism,
and write $G=G(P)$ and $G'=G(Q)$.
Since poset isomorphisms preserve direct sums and ordinal sums,
the recursive definition of $\mathrm{SP}$ implies that $Q\in\mathrm{SP}$.
For every $U\subseteq P$, set $\varphi(U)=\{\varphi(u):u\in U\}$.
By the definition of a poset isomorphism,
the map $u\mapsto\varphi(u)$ defines a poset isomorphism from $P[U]$ to $Q[\varphi(U)]$
and induces graph isomorphisms $G[U]\cong G'[\varphi(U)]$ and
$\overline G[U]\cong\overline{G'}[\varphi(U)]$.
Let $\mathcal K'=\{\varphi(U):U\in\mathcal W(P)\}$.
Then $\mathcal K'\subseteq 2^Q\setminus\{\varnothing\}$ and $Q=\varphi(P)\in\mathcal K'$.
Take any $U\in\mathcal W(P)$ and any $D\in\mathcal C(G'[\varphi(U)])\cup\mathcal C(\overline{G'}[\varphi(U)])$.
By the above graph isomorphisms, there exists $C\in\mathcal C(G[U])\cup\mathcal C(\overline G[U])$
such that $V(D)=\varphi(V(C))$.
The definition of $\mathcal W(P)$ gives $V(C)\in\mathcal W(P)$, and hence $V(D)\in\mathcal K'$.
Therefore, $\mathcal K'$ contains $Q$ and satisfies the closure condition in the definition of $\mathcal W(Q)$.
It follows that $\mathcal W(Q)\subseteq\mathcal K'$.
Applying the same argument to $\varphi^{-1}$ gives $\mathcal W(P)\subseteq\{\varphi^{-1}(X):X\in\mathcal W(Q)\}$.
Taking images under $\varphi$, we obtain $\mathcal K'\subseteq\mathcal W(Q)$.
Thus,% $\mathcal W(Q)=\mathcal K'=\{\varphi(U):U\in\mathcal W(P)\}$.
\begin{eqnarray*}
\mathcal W(Q)=\mathcal K'=\{\varphi(U):U\in\mathcal W(P)\}.
\end{eqnarray*}

Now assume that $P\in\mathrm H$.
Take any $X\in\mathcal W(Q)$ such that $G'[X]$ is disconnected,
and take any $D\in\mathcal C(G'[X])$.
By the equality of families established above,
there exists $U\in\mathcal W(P)$ such that $X=\varphi(U)$.
Since $\varphi$ induces a graph isomorphism
$G[U]\cong G'[X]$, the graph $G[U]$ is disconnected,
and there exists $C\in\mathcal C(G[U])$ such that
$V(D)=\varphi(V(C))$.
As $P\in\mathrm H$, the polynomial $T(P[V(C)];x,y)$
is irreducible in $\mathbb Q[x,y]$.
Moreover, the map $u\mapsto\varphi(u)$ defines a poset isomorphism
from $P[V(C)]$ to $Q[V(D)]$.
It maps antichains bijectively to antichains and preserves
both their sizes and the sizes of the order filters
they generate in the respective induced subposets.
By Lemma \ref{lem:antichain}, we therefore have
$T(Q[V(D)];x,y)=T(P[V(C)];x,y)$.
Hence, $T(Q[V(D)];x,y)$ is irreducible in $\mathbb Q[x,y]$.
Since $X$ and $D$ were arbitrary and $Q\in\mathrm{SP}$,
it follows that $Q\in\mathrm H$.
Applying the same argument to the inverse isomorphism $\varphi^{-1}$
shows that $Q\in\mathrm H$ implies $P\in\mathrm H$.
Therefore, $P\in\mathrm H$ if and only if $Q\in\mathrm H$,
so $\mathrm H$ is invariant under poset isomorphism.

We now prove that $P\in\mathrm H$ and $W\in\mathcal W(P)$
imply $P[W]\in\mathrm H$. Fix such $P$ and $W$, and write $G=G(P)$.
%We now prove that $P\in\mathrm H$ and $W\in\mathcal W(P)$ imply $P[W]\in\mathrm H$.
Since $P\in\mathrm H\subseteq\mathrm{SP}$,
part \textup{(i)} of Lemma \ref{lem:graph-decomposition} gives $P[W]\in\mathrm{SP}$.
Take any $U\in\mathcal W(P[W])$, and suppose that $G(P[W])[U]$ is disconnected.
By the formula \eqref{eq:W-inheritance}, we have $U\in\mathcal W(P)$ and $U\subseteq W$.
The definitions of an induced subposet and an induced subgraph give $G(P[W])[U]=(G[W])[U]=G[U]$.
Thus, $G[U]$ is disconnected and has the same connected components as $G(P[W])[U]$.
Take any $C\in\mathcal C(G(P[W])[U])$.
By the above graph identity, we have $C\in\mathcal C(G[U])$.
Since $P\in\mathrm H$, $U\in\mathcal W(P)$,
and $G[U]$ is disconnected, Definition \ref{def:poset-h}
implies that $T(P[V(C)];x,y)$ is irreducible in $\mathbb Q[x,y]$.
Also, $V(C)\subseteq U\subseteq W$, so the definition
of an induced subposet gives $(P[W])[V(C)]=P[V(C)]$.
Hence, $T((P[W])[V(C)];x,y)=T(P[V(C)];x,y)$,
and therefore $T((P[W])[V(C)];x,y)$ is irreducible in $\mathbb Q[x,y]$.
The above argument, together with $P[W]\in\mathrm{SP}$,
shows that $P[W]$ satisfies all the conditions of Definition \ref{def:poset-h}.
Consequently, $P[W]\in\mathrm H$.

Finally, let $P\in\mathrm H$ and write $G=G(P)$.
Since $P\in\mathcal W(P)$, $G=G[P]$, and $\overline G=\overline G[P]$,
taking $W=P$ in the definition of $\mathcal W(P)$ shows that $V(C)\in\mathcal W(P)$ 
for every $C\in\mathcal C(G)\cup\mathcal C(\overline G)$.
By the inheritance property proved above, we obtain $P[V(C)]\in\mathrm H$.
By the decomposition results stated earlier, the factors in the decomposition of $P$ 
as a direct sum of nonempty factors that are indecomposable with respect
to direct sums are precisely the posets $P[V(C)]$, where $C\in\mathcal C(G)$.
Similarly, the factors in the decomposition of $P$ as an ordinal sum 
of nonempty ordinally indecomposable factors are precisely the posets $P[V(D)]$,
where $D\in\mathcal C(\overline G)$,
and their order is uniquely determined by the order relation of $P$.
Therefore, every factor in either decomposition belongs to $\mathrm H$.

If $c(G)\geq 2$, then $G[P]$ is disconnected. Since $P\in\mathcal W(P)$ and $P\in\mathrm H$,
taking $W=P$ in Definition \ref{def:poset-h} shows that,  for every $C\in\mathcal C(G)$,
the polynomial $T(P[V(C)];x,y)$ is irreducible in $\mathbb Q[x,y]$.
Consequently, in the decomposition of $P$ as a direct sum of nonempty factors 
that are indecomposable with respect to direct sums, 
the Tutte polynomial of every factor is irreducible in $\mathbb Q[x,y]$.
This completes the proof of the lemma.
\end{proof}

\begin{lemma}(\cite{Gordon1996})\label{lem:leantichain}
A poset $R$ is an ordinal sum of two non-empty posets if and only if 
there is some positive integer $k < |R|$ such that 
there is exactly one order ideal of size $k$ in $R$.
\end{lemma}

Lemma \ref{lem:leantichain} yields the following analogous
characterization in terms of order filters.

\begin{lemma}\label{lem:unique-filter}
A poset $R$ is an ordinal sum of two nonempty posets if and only if 
there is some positive integer $k<|R|$ such that 
there is exactly one order filter of size $k$ in $R$.
\end{lemma}

\begin{proof}
Let $R^*$ be the dual poset of $R$. By Definition \ref{def:poset-sums-dual}, 
we have $x\le y$ in $R^*$ if and only if $y\le x$ in $R$.
By the definitions of an order ideal and an order filter,
a subset $F\subseteq R$ is an order filter of $R$
if and only if the same subset, viewed as a subset of $R^*$,
is an order ideal of $R^*$.
Therefore, $R$ has exactly one order filter of size $k$
if and only if $R^*$ has exactly one order ideal of size $k$.

Also, according to Lemma \ref{lem:leantichain} and the above arguments, 
the statement that there exists a positive integer $k<|R|$ such that 
$R$ has exactly one order filter of size $k$ is equivalent to 
the statement that $R^*$ is an ordinal sum of two nonempty posets.
Such a decomposition can be written as $R^*=Q^*\oplus P^*$ 
for some nonempty posets $P^*$ and $Q^*$.
Since $Q^*\oplus P^*=(P\oplus Q)^*$ and $(R^*)^*=R$,
we have $R^*=Q^*\oplus P^*$ if and only if $R=P\oplus Q$.
Hence, $R^*$ is an ordinal sum of two nonempty posets
if and only if $R$ is an ordinal sum of two nonempty posets.
Combining these equivalences, we conclude that $R$ is an ordinal sum
of two nonempty posets if and only if there exists a positive integer
$k<|R|$ such that $R$ has exactly one order filter of size $k$.
\end{proof}

For every finite poset $P$ and every integer
$0\leq k\leq |P|$, let $b_k(P)$ denote the coefficient of
$t^k$ in $T(P;t+1,1)$. Equivalently,
\begin{eqnarray*}
T(P;t+1,1)=\sum_{k=0}^{|P|}b_k(P)t^k.
\end{eqnarray*}

\begin{lemma}\label{lem:connectivity}
Let $P\in\mathrm{SP}$ and $n=|P|\geq 2$.
Then $c(G(P))=1$ if and only if there exists an integer
$0<k<n$ such that the coefficient of $t^k$ in $T(P;t+1,1)$
is $1$.
Consequently, two nonempty series-parallel posets with the same
Tutte polynomial are either both connected or both disconnected.
\end{lemma}

\begin{proof}
We first derive a formula for the coefficient of $t^k$ in $T(P;t+1,1)$.
For a digraph $D$ and a subset $X\subseteq V(D)$, let
\begin{eqnarray*}
\delta_D^+(X)=\{(u,v)\in A(D):u\in X,\ v\in V(D)\setminus X\}
\end{eqnarray*}
denote the outgoing cut of $D$ associated with $X$.
Since $A(\vec G(P))=\{(u,v)\in P\times P:u<v\}$, for every $F\subseteq P$, 
we have $\delta_{\vec G(P)}^+(F)=\varnothing$ if and only if
there do not exist $u\in F$ and $v\in P\setminus F$ such that $u<v$.
By the definition of an order filter, this is equivalent
to $F$ being an order filter of $P$.

Let $n=|P|$.
For each integer $0\leq k\leq n$. %let $b_k(P)$ denote the coefficient of $t^k$ in $T(P;t+1,1)$.
Substituting $x=t+1$ and $y=1$ into the formula \eqref{eq:antichain-expansion}, we obtain
\begin{eqnarray*}
T(P;t+1,1)=\sum_{A\in\mathcal A(P)}t^{|I^*(A)|}.
\end{eqnarray*}
%where $\mathcal A(P)$ denotes the set of all antichains of $P$.
Each antichain $A$ contributes $1$ to the coefficient of $t^k$
if and only if $|I^*(A)|=k$, and contributes $0$ otherwise.
Thus, $b_k(P)$ equals the number of antichains $A$ satisfying $|I^*(A)|=k$.
However, the map $A\mapsto I^*(A)$ is a bijection from the set
of antichains of $P$ to the set of order filters of $P$.
Under this bijection, the antichains satisfying $|I^*(A)|=k$
correspond precisely to the order filters of size $k$.
Hence, $b_k(P)$ equals the number of order filters of $P$ of size $k$.
Moreover, since $A(\vec G(P))=\{(u,v)\in P\times P:u<v\}$,
a subset $F\subseteq P$ is an order filter if and only if
there is no arc from $F$ to $P\setminus F$,
that is, if and only if $\delta_{\vec G(P)}^+(F)=\varnothing$. Therefore,
\begin{equation}\label{eq:b-coefficients}
b_k(P)=\bigl|\{F\subseteq P:|F|=k,\delta_{\vec G(P)}^+(F)=\varnothing\}\bigr|.
\end{equation}
In particular, $\varnothing$ and $P$ are order filters
and are the unique subsets of $P$ of sizes $0$ and $n$, respectively.
Thus, $b_0(P)=b_n(P)=1$.

We now prove the lemma. 
First, suppose that $P\in\mathrm{SP}$, $n=|P|\geq 2$, and $c(G(P))=1$.
By part \textup{(iii)} of Lemma \ref{lem:graph-decomposition},
the connected components of $\overline{G(P)}$ admit a unique ordering $D_1,\ldots,D_r$, 
where $r\geq2$, such that $P=P[V(D_1)]\oplus\cdots\oplus P[V(D_r)]$.
Let $Y=V(D_r)$ and $X=P\setminus Y$. Then $X$ and $Y$ are both nonempty.
By the above ordinal sum decomposition, we have $u<v$ for every $u\in X$ and every $v\in Y$.
Hence, $P=P[X]\oplus P[Y]$. Set $k=|Y|$, so that $0<k<n$.
Hence, by Lemma \ref{lem:unique-filter} and the above arguments,
there exists an integer $k$ with $0<k<n$
such that $P$ has exactly one order filter of size $k$.
It follows from the formula \eqref{eq:b-coefficients} that $b_k(P)=1$.

Conversely, suppose that there exists an integer $0<k<n$ such that $b_k(P)=1$.
By the formula \eqref{eq:b-coefficients}, the poset $P$ has exactly one order filter of size $k$.
Lemma \ref{lem:unique-filter} implies that there exist nonempty posets $A$ 
and $B$ with disjoint ground sets such that $P=A\oplus B$.
By the formula \eqref{eq:graph-sums}, we have $G(P)=G(A)\vee G(B)$.
Thus, $\{a,b\}\in E(G(P))$ for every $a\in A$ and every $b\in B$.
Choose $a_0\in A$, and let $C\in\mathcal C(G(P))$ be the connected component containing $a_0$.
Since every vertex of $B$ is adjacent to $a_0$, we have $B\subseteq V(C)$.
As $B\neq\varnothing$, we may choose $b_0\in B$.
Then $b_0\in V(C)$, and every vertex of $A$ is adjacent to $b_0$, so $A\subseteq V(C)$.
Therefore, by the above arguments, we have $A\cup B\subseteq V(C)$.
Since the ground set of $P$ is $A\cup B$, the graph $G(P)$ has only one connected component.
Hence, $c(G(P))=1$, proving the stated equivalence.

Finally, to prove that nonempty series-parallel posets
with the same Tutte polynomial have the same connectivity,
we first show that the polynomial determines the size of the ground set.
Let $P$ be an arbitrary nonempty finite poset, and write $n=|P|$.
By the formula \eqref{eq:antichain-expansion}, the term corresponding to an antichain $A$ is
$(x-1)^{|I^*(A)|}y^{|I^*(A)|-|A|}$.
If $A\neq\varnothing$, then $A\subseteq I^*(A)$ implies
$|I^*(A)|\geq 1$, so this term vanishes at $x=1$.
If $A=\varnothing$, then $I^*(A)=\varnothing$,
and the corresponding term is identically $1$.
Therefore, $T(P;1,y)=1$.

Since $I^*(A)\subseteq P$, every term in the above expansion has degree at most $n$ in $x$.
Hence, $\deg_x T(P;x,y)\leq n$.
By the definition of $b_k(P)$, we have $T(P;x,1)=\sum_{k=0}^{n}b_k(P)(x-1)^k$.
Since $b_n(P)=1$ and $(x-1)^k$ contains no $x^n$ term when $k<n$, 
the coefficient of $x^n$ in $T(P;x,1)$ is $1$.
Consequently, the coefficient of $x^n$ in $T(P;x,y)$,
viewed as a polynomial in $y$, takes the value $1$ at $y=1$ and is therefore not the zero polynomial.
Thus, $\deg_x T(P;x,y)\geq n$.
According to the above arguments, we obtain
\begin{equation}\label{eq:normalization-degree}
T(P;1,y)=1,\qquad \deg_x T(P;x,y)=|P|.
\end{equation}
In particular, since $T(P;1,y)=1$ and $\deg_x T(P;x,y)=|P|\geq 1$,
the Tutte polynomial of a nonempty finite poset is nonzero and nonconstant.

Now suppose that nonempty series-parallel posets $P$ and $Q$
satisfy $T(P;x,y)=T(Q;x,y)$.
By the formula \eqref{eq:normalization-degree}, we have $|P|=|Q|$.
Let $n$ denote their common size.
Substituting $x=t+1$ and $y=1$ into the polynomial identity
and comparing coefficients of $t^k$, we obtain
$b_k(P)=b_k(Q)$ for every integer $0\leq k\leq n$.
Thus, if $n\geq 2$, there exists an integer $0<k<n$
such that $b_k(P)=1$ if and only if there exists an integer
$0<k<n$ such that $b_k(Q)=1$.
Applying the connectivity criterion proved above to $P$ and $Q$,
we conclude that $c(G(P))=1$ if and only if $c(G(Q))=1$.
If $n=1$, then $G(P)$ and $G(Q)$ are both graphs
with a single vertex and are therefore connected.
Hence, nonempty series-parallel posets with the same Tutte polynomial
are either both connected or both disconnected.
\end{proof}

\begin{lemma}(\cite{Gordon1993})\label{lem:sum-formulas}
Let $P$ and $Q$ be finite posets with disjoint ground sets. Then
\begin{equation}\label{eq:direct-sum}
T(P+Q;x,y)=T(P;x,y)T(Q;x,y).
\end{equation}
\end{lemma}

\begin{lemma}(\cite{Gordon1996})\label{lem:ordinal-recovery}
Let $P$ and $Q$ be ordinally indecomposable posets.
If a poset $R$ satisfies $T(R;x,y)=T(P\oplus Q;x,y)$,
then $R=P'\oplus Q'$ for some posets $P'$ and $Q'$
such that $T(P;x,y)=T(P';x,y)$ and $T(Q;x,y)=T(Q';x,y)$.
\end{lemma}

\begin{lemma}\label{lem:no-grouping}
Suppose that $P\in\mathrm{SP}$ and $c(G(P))=1$.
If $U_1,\ldots,U_m$ are nonempty posets and $m\geq 2$, then
\begin{eqnarray*}
T(P;x,y)\neq\prod_{i=1}^{m}T(U_i;x,y).
\end{eqnarray*}
\end{lemma}

\begin{proof}
Suppose, to the contrary, that $T(P;x,y)=\prod_{i=1}^{m}T(U_i;x,y)$.
Since the Tutte polynomial is invariant under poset isomorphism,
we may assume that the ground sets of $U_1,\ldots,U_m$ are pairwise disjoint.
Let $U=U_1+\cdots+U_m$.
Repeated application of the formula \eqref{eq:direct-sum} gives $T(P;x,y)=T(U;x,y)$.
By the formula \eqref{eq:normalization-degree} and the preceding discussion, we obtain
\begin{eqnarray*}
n=|P|=|U|=\sum_{i=1}^{m}|U_i|\geq m\geq 2.
\end{eqnarray*}
Since $P\in\mathrm{SP}$ and $G(P)$ is connected,
Lemma \ref{lem:connectivity} implies that there exists an integer $0<k<n$ such that $b_k(P)=1$.
Substituting $x=t+1$ and $y=1$ into $T(P;x,y)=T(U;x,y)$ and comparing coefficients of $t^k$,
we obtain $b_k(U)=b_k(P)=1$.
By the formula \eqref{eq:b-coefficients}, the equality $b_k(U)=1$ implies that $U$ has exactly one order filter of size $k$.
Lemma \ref{lem:unique-filter} then gives nonempty posets $A$ and $B$ such that $U=A\oplus B$.
By the formula \eqref{eq:graph-sums} and the preceding discussion, we have $G(U)=G(A)\vee G(B)$.

Take any $a\in A$ and $b\in B$.
The identity $G(U)=G(A)\vee G(B)$ implies that $\{a,b\}\in E(G(U))$.
Thus, $a$ and $b$ belong to the same connected component $C$.
Since every vertex of $A$ is adjacent to $b$ and every vertex of $B$ is adjacent to $a$,
we have $V(C)=A\cup B=U$. Hence, $c(G(U))=1$.

On the other hand, $U=U_1+\cdots+U_m$ and the formula \eqref{eq:graph-sums} give $G(U)=\bigcup_{i=1}^{m}G(U_i)$.
Since every $U_i$ is nonempty, we obtain
\begin{eqnarray*}
c(G(U))=\sum_{i=1}^{m}c(G(U_i))\geq m\geq 2,
\end{eqnarray*}
contradicting $c(G(U))=1$. Consequently, $T(P;x,y)\neq\prod_{i=1}^{m}T(U_i;x,y)$.
\end{proof}

\section{Main results}

%In this section, we continue to use the notation introduced in the preceding
%section. We  prove that $\mathcal P\subseteq\mathrm H$.
%And we prove the main theorem \ref{thm:main} using the lemmas established there, 
%thereby obtaining Gordon's isomorphism characterization for the subclass $\mathcal P$
%of series-parallel posets \cite{Gordon1996} as a special case
%of the main theorem.
In this section, we prove the inclusion $\mathcal P\subseteq\mathrm H$ and establish
Theorem \ref{thm:main} using the lemmas from Section 2.

\begin{proof}[\rm\textbf{Proof of Theorem \ref{thm:main-sub}}]
%Since $\mathcal P\subseteq\mathrm{SP}$, by the main result,
%Theorem \ref{thm:main}, it suffices to prove that
%$\mathcal P\subseteq\mathrm H$.
To prove the theorem, we establish the following claim:

\textbf{Claim:}
For every $P\in\mathcal P$, every $W\in\mathcal W(P)$, and every $C\in\mathcal C(G(P[W]))$,
the polynomial $T(P[V(C)];x,y)$ is irreducible in $\mathbb Q[x,y]$.

By the definition of an induced subposet, for every nonempty $W\subseteq P$, 
we have $G(P[W])=G(P)[W]$ and $\overline{G(P[W])}=\overline{G(P)}[W]$.
We first reformulate the claim in an equivalent form that is convenient for induction:
for every $P\in\mathcal P$ and every $W\in\mathcal W(P)$,
if $G(P)[W]$ is connected, then $T(P[W];x,y)$ is irreducible in $\mathbb Q[x,y]$.

Indeed, if the claim holds and $G(P)[W]$ is connected,
then $G(P)[W]$ itself is its only connected component, with vertex set $W$.
The claim therefore directly implies that $T(P[W];x,y)$ is irreducible in $\mathbb Q[x,y]$.
Conversely, suppose that whenever $G(P)[W]$ is connected,
the polynomial $T(P[W];x,y)$ is irreducible in $\mathbb Q[x,y]$.
Take any $W\in\mathcal W(P)$ and $C\in\mathcal C(G(P)[W])$, and set $U=V(C)$.
Since $W\in\mathcal W(P)$ and $C$ is a connected component of the induced subgraph $G(P)[W]$,
the definition of $\mathcal W(P)$ implies that the vertex set $V(C)$ also belongs to $\mathcal W(P)$.
As $U=V(C)$, we have $U\in\mathcal W(P)$.
Since a connected component is the connected subgraph induced by its vertex set, 
we have $G(P)[U]=(G(P)[W])[U]=C$.
Thus, $G(P)[U]$ is connected.
Applying the assumption to $U$ and using the preceding discussion,
we conclude that $T(P[V(C)];x,y)=T(P[U];x,y)$ is irreducible.
Therefore, to prove the claim, it suffices to establish the equivalent formulation above.

To prove this equivalent formulation,
we first establish two formulas that will be needed.
Let $P,Q\in\mathrm{SP}$ be nonempty and disjoint.
By the definitions of the direct sum and the ordinal sum, we have
$G(P+Q)=G(P)\cup G(Q)$, $\overline{G(P+Q)}=\overline{G(P)}\vee\overline{G(Q)}$,
$G(P\oplus Q)=G(P)\vee G(Q)$, and $\overline{G(P\oplus Q)}=\overline{G(P)}\cup\overline{G(Q)}$.
We now prove that
\begin{align}
\mathcal W(P+Q)
&\subseteq\{P\cup Q\}\cup\mathcal W(P)\cup\mathcal W(Q),\label{eq:cor3343-W-direct}\\
\mathcal W(P\oplus Q)
&\subseteq\{P\cup Q\}\cup\mathcal W(P)\cup\mathcal W(Q).\label{eq:cor3343-W-ordinal}
\end{align}

For these disjoint nonempty posets $P,Q\in\mathrm{SP}$,
set $\mathcal K=\{P\cup Q\}\cup\mathcal W(P)\cup\mathcal W(Q)$.
Since $\mathcal W(P)\subseteq 2^P\setminus\{\varnothing\}$
and $\mathcal W(Q)\subseteq 2^Q\setminus\{\varnothing\}$,
we have $\mathcal K\subseteq 2^{P\cup Q}\setminus\{\varnothing\}$.
Also, the definition of $\mathcal K$ gives $P\cup Q\in\mathcal K$.

We next verify that $\mathcal K$ satisfies the conditions
required in the definitions of $\mathcal W(P+Q)$ and $\mathcal W(P\oplus Q)$.
Take any $U\in\mathcal K$, let $F$ be any one of the graphs
$G(P+Q)$, $\overline{G(P+Q)}$, $G(P\oplus Q)$, and $\overline{G(P\oplus Q)}$,
and take any $C\in\mathcal C(F[U])$.
We prove that $V(C)\in\mathcal K$.
According to the definition of $\mathcal K$, we consider the following three cases:

\textbf{Case 1:} $U=P\cup Q$.

In this case, $U=P\cup Q=V(F)$, so $F[U]=F$ and hence $C\in\mathcal C(F)$.
If $F=\overline{G(P+Q)}$ or $F=G(P\oplus Q)$,
then $F=\overline{G(P)}\vee\overline{G(Q)}$ or $F=G(P)\vee G(Q)$, respectively.
Since $P$ and $Q$ are both nonempty, both joins are connected.
Therefore, $C=F$, and hence $V(C)=P\cup Q\in\mathcal K$.

If $F=G(P+Q)$ or $F=\overline{G(P\oplus Q)}$, then $F=G(P)\cup G(Q)$
or $F=\overline{G(P)}\cup\overline{G(Q)}$, respectively.
By the decomposition of a disjoint union into connected components,
$C$ is a connected component of $G(P)$, $\overline{G(P)}$, $G(Q)$, or $\overline{G(Q)}$.
If $C\in\mathcal C(G(P))\cup\mathcal C(\overline{G(P)})$,
then $P\in\mathcal W(P)$ and the definition of $\mathcal W(P)$
give $V(C)\in\mathcal W(P)\subseteq\mathcal K$.
If $C\in\mathcal C(G(Q))\cup\mathcal C(\overline{G(Q)})$,
then similarly, $Q\in\mathcal W(Q)$ and the definition of $\mathcal W(Q)$ give
$V(C)\in\mathcal W(Q)\subseteq\mathcal K$.

Thus, $V(C)\in\mathcal K$ in all cases.

\textbf{Case 2:} $U\in\mathcal W(P)$.

In this case, $U\subseteq P$.
Since neither the direct sum nor the ordinal sum changes
the order relations between elements of $P$,
we have $(P+Q)[U]=(P\oplus Q)[U]=P[U]$.
Therefore, $G(P+Q)[U]=G(P\oplus Q)[U]=G(P)[U]$ and
$\overline{G(P+Q)}[U]=\overline{G(P\oplus Q)}[U]=\overline{G(P)}[U]$.
By the choice of $F$ and the preceding discussion,
$F[U]$ equals either $G(P)[U]$ or $\overline{G(P)}[U]$.
Hence, $C\in\mathcal C(G(P)[U])\cup\mathcal C(\overline{G(P)}[U])$.
Since $U\in\mathcal W(P)$, the definition of $\mathcal W(P)$ implies that
the vertex set $V(C)$ of this connected component belongs to $\mathcal W(P)$.
Therefore, $V(C)\in\mathcal W(P)\subseteq\mathcal K$.

\textbf{Case 3:} $U\in\mathcal W(Q)$.

In this case, $U\subseteq Q$, and $(P+Q)[U]=(P\oplus Q)[U]=Q[U]$.
Therefore, $G(P+Q)[U]=G(P\oplus Q)[U]=G(Q)[U]$ and
$\overline{G(P+Q)}[U]=\overline{G(P\oplus Q)}[U]=\overline{G(Q)}[U]$.
Thus, $F[U]$ equals either $G(Q)[U]$ or $\overline{G(Q)}[U]$, 
and hence $C\in\mathcal C(G(Q)[U])\cup\mathcal C(\overline{G(Q)}[U])$.
Since $U\in\mathcal W(Q)$, the definition of $\mathcal W(Q)$ gives
$V(C)\in\mathcal W(Q)\subseteq\mathcal K$.

It follows that, for every $U\in\mathcal K$,
the vertex set of every connected component of
$G(P+Q)[U]$, $\overline{G(P+Q)}[U]$,
$G(P\oplus Q)[U]$, and $\overline{G(P\oplus Q)}[U]$
belongs to $\mathcal K$.
Together with $P\cup Q\in\mathcal K$,
this shows that $\mathcal K$ satisfies all the conditions
in the definitions of both $\mathcal W(P+Q)$
and $\mathcal W(P\oplus Q)$.
Since these two families are the smallest with respect
to inclusion satisfying the respective conditions,
we obtain $\mathcal W(P+Q)\subseteq\mathcal K$
and $\mathcal W(P\oplus Q)\subseteq\mathcal K$.
This proves the formulas \eqref{eq:cor3343-W-direct}
and \eqref{eq:cor3343-W-ordinal}.

Moreover, the dual poset $P^*$ has the same ground set as $P$,
and two elements are comparable in $P^*$
if and only if they are comparable in $P$.
Therefore, $G(P^*)=G(P)$ and
$\overline{G(P^*)}=\overline{G(P)}$.
Thus, the ground sets, induced subgraphs, and connected components
used in the definitions of $\mathcal W(P^*)$ and $\mathcal W(P)$
are the same.
Since these two families are the smallest with respect
to inclusion satisfying their defining conditions,
we obtain $\mathcal W(P^*)=\mathcal W(P)$.

We first discuss the irreducibility results needed in the proof.
Let $P$ be a nonempty finite poset.
%After the substitution $y=z+1$, the polynomial used satisfies
%$f(P;t,y)=T(P;t+1,y)$ and $f(P;0,y)=1$.
%Since $f(P;0,y)=1$, the constant term of $f(P;t,y)$ is $1$,
%so it is a primitive polynomial with integer coefficients.
%By Gauss's lemma, $f(P;t,y)$ is irreducible in $\mathbb Z[t,y]$
%if and only if it is irreducible in $\mathbb Q[t,y]$.
We now recall the irreducibility results needed below.
Let $P$ be a nonempty finite poset, and define
\begin{eqnarray*}
f(P;t,y)=T(P;t+1,y).
\end{eqnarray*}
This agrees with Gordon's notation after the change of variables used in \cite{Gordon1996}.
By the formula \eqref{eq:normalization-degree}, we have $f(P;0,y)=1$.
Thus, $f(P;t,y)$ has integer coefficients and constant
coefficient $1$, so it is primitive.
Gauss's lemma and the invertible substitution $t=x-1$
give
\begin{eqnarray*}
\begin{aligned}
&f(P;t,y)\text{ is irreducible in }\mathbb Z[t,y]\\
&\quad\text{if and only if}\quad
f(P;t,y)\text{ is irreducible in }\mathbb Q[t,y]\\
&\quad\text{if and only if}\quad
T(P;x,y)\text{ is irreducible in }\mathbb Q[x,y].
\end{aligned}
\end{eqnarray*}
%Moreover, the substitution $t=x-1$ and its inverse $x=t+1$
%define a ring isomorphism between $\mathbb Q[t,y]$
%and $\mathbb Q[x,y]$ that maps $f(P;t,y)$ to $T(P;x,y)$.
%Since ring isomorphisms preserve irreducibility,
%$f(P;t,y)$ is irreducible in $\mathbb Z[t,y]$
%f and only if $T(P;x,y)$ is irreducible in $\mathbb Q[x,y]$.
The result in \cite{Gordon1996} states that $f(P;t,y)$
is irreducible in $\mathbb Z[t,y]$ if and only if
$f(P^*;t,y)$ is irreducible in the same ring.
Applying the above correspondence to $P$ and $P^*$,
we conclude that $T(P;x,y)$ is irreducible in $\mathbb Q[x,y]$
if and only if $T(P^*;x,y)$ is irreducible in the same ring.

Now let $\mathbf n$ be an antichain with $n\geq1$ elements,
and suppose that its ground set is disjoint from that of $P$.
Since $P^*$ is nonempty, the result in \cite{Gordon1996}
implies that $f(\mathbf n\oplus P^*;t,y)$
is irreducible in $\mathbb Z[t,y]$.
Applying the above correspondence to the poset
$\mathbf n\oplus P^*$, we obtain that
$T(\mathbf n\oplus P^*;x,y)$
is irreducible in $\mathbb Q[x,y]$.
Since $\mathbf n^*\cong\mathbf n$, we have
$(P\oplus\mathbf n)^*\cong\mathbf n\oplus P^*$.
Thus, the invariance of the Tutte polynomial under poset isomorphism
gives
$T((P\oplus\mathbf n)^*;x,y)=T(\mathbf n\oplus P^*;x,y)$,
so $T((P\oplus\mathbf n)^*;x,y)$ is irreducible.
Finally, applying the above result on irreducibility under duality
to the poset $P\oplus\mathbf n$, we conclude that
$T(P\oplus\mathbf n;x,y)$ is irreducible in $\mathbb Q[x,y]$.

We now prove the equivalent formulation of the claim
by structural induction on the four construction rules
defining $\mathcal P$.
This condition is invariant under poset isomorphism.
Indeed, if $\varphi:P\to Q$ is a poset isomorphism,
then, since $\mathcal W(P)$ and $\mathcal W(Q)$
are the smallest families satisfying their respective
defining conditions, we have
$\mathcal W(Q)=\{\varphi(W):W\in\mathcal W(P)\}$.
Moreover, for every $W\in\mathcal W(P)$, we have
$G(P)[W]\cong G(Q)[\varphi(W)]$
and $P[W]\cong Q[\varphi(W)]$.
Therefore, the connectivity of the corresponding induced subgraphs
and the irreducibility of the corresponding Tutte polynomials
are both preserved.
By the recursive definition of $\mathcal P$,
it suffices to show that $\mathbf1$ satisfies the equivalent
formulation of the claim and that the other three construction rules
preserve this condition.
In each induction step below, we assume that the posets used
in the current construction already satisfy this condition.
We consider the four cases as follows:

\textbf{Case 1:} The initial poset is $\mathbf1$.

In this case, $\mathcal W(\mathbf1)=\{\mathbf1\}$,
and $G(\mathbf1)$ is a graph with a single vertex.
Since $T(\mathbf1;x,y)=x$ is irreducible,
the equivalent formulation of the claim holds.

\textbf{Case 2:} Constructing $P\oplus\mathbf n$ from $P$,
where $n\geq1$.

Suppose that $P\in\mathcal P$ satisfies the induction hypothesis,
and take an antichain $\mathbf n$ whose ground set is disjoint
from that of $P$.
Take any $W\in\mathcal W(P\oplus\mathbf n)$ such that
$G(P\oplus\mathbf n)[W]$ is connected.
We show that $T((P\oplus\mathbf n)[W];x,y)$ is irreducible.

If $W=P\cup\mathbf n$, then
$(P\oplus\mathbf n)[W]=P\oplus\mathbf n$,
so the required irreducibility follows
from the result established above.

If $W\neq P\cup\mathbf n$, then the formula \eqref{eq:cor3343-W-ordinal}
gives $W\in\mathcal W(P)$ or $W\in\mathcal W(\mathbf n)$.
If $W\in\mathcal W(P)$, then
$(P\oplus\mathbf n)[W]=P[W]$,
and hence $G(P)[W]=G(P\oplus\mathbf n)[W]$ is connected.
By the induction hypothesis for $P$,
$T((P\oplus\mathbf n)[W];x,y)=T(P[W];x,y)$ is irreducible.

If $W\in\mathcal W(\mathbf n)$, then
$(P\oplus\mathbf n)[W]=\mathbf n[W]$.
Since $\mathbf n$ is an antichain,
$G(\mathbf n)[W]$ is an edgeless graph.
As $G(\mathbf n)[W]=G(P\oplus\mathbf n)[W]$ is connected
and $W\neq\varnothing$, we must have $|W|=1$.
Thus, $(P\oplus\mathbf n)[W]\cong\mathbf1$, and hence
$T((P\oplus\mathbf n)[W];x,y)=x$ is irreducible.

Therefore, $P\oplus\mathbf n$ satisfies
the equivalent formulation of the claim.

\textbf{Case 3:} Constructing $P+Q$ from $P$ and $Q$.

Suppose that $P,Q\in\mathcal P$ both satisfy the induction hypothesis
and have disjoint ground sets.
Take any $W\in\mathcal W(P+Q)$ such that
$G(P+Q)[W]$ is connected.
Since $P$ and $Q$ are both nonempty,
$G(P+Q)=G(P)\cup G(Q)$ is disconnected,
so $W\neq P\cup Q$.
By the formula \eqref{eq:cor3343-W-direct},
we have $W\in\mathcal W(P)$ or $W\in\mathcal W(Q)$.

If $W\in\mathcal W(P)$, then $(P+Q)[W]=P[W]$,
and hence $G(P)[W]=G(P+Q)[W]$ is connected.
By the induction hypothesis for $P$,
$T((P+Q)[W];x,y)=T(P[W];x,y)$ is irreducible.
If $W\in\mathcal W(Q)$, then $(P+Q)[W]=Q[W]$,
and $G(Q)[W]=G(P+Q)[W]$ is connected.
By the induction hypothesis for $Q$,
$T((P+Q)[W];x,y)=T(Q[W];x,y)$ is irreducible.
Therefore, $P+Q$ satisfies the equivalent formulation of the claim.

\textbf{Case 4:} Constructing $P^*$ from $P$.

Suppose that $P\in\mathcal P$ satisfies the induction hypothesis.
Take any $W\in\mathcal W(P^*)$ such that $G(P^*)[W]$ is connected.
The identities $\mathcal W(P^*)=\mathcal W(P)$
and $G(P^*)=G(P)$ established above imply that
$W\in\mathcal W(P)$ and $G(P)[W]$ is connected.
Thus, by the induction hypothesis for $P$,
$T(P[W];x,y)$ is irreducible.
Since $(P^*)[W]=(P[W])^*$,
applying the above result on irreducibility under duality
to the nonempty poset $P[W]$ shows that
$T((P^*)[W];x,y)$ is irreducible.
Therefore, $P^*$ satisfies the equivalent formulation of the claim.

We have shown that $\mathbf1$ satisfies
the equivalent formulation of the claim
and that the other three construction rules preserve this condition.
By induction, the equivalent formulation holds
for every $P\in\mathcal P$.
Since this condition is equivalent to the original claim,
the claim follows.

Now take any $P\in\mathcal P$.
For every $W\in\mathcal W(P)$, if $G(P[W])$ is disconnected,
then, for every $C\in\mathcal C(G(P[W]))$,
the polynomial $T(P[V(C)];x,y)$
is irreducible in $\mathbb Q[x,y]$.
Together with $P\in\mathrm{SP}$,
Definition \ref{def:poset-h} implies that $P\in\mathrm H$.
Since $P$ was arbitrary, we obtain $\mathcal P\subseteq\mathrm H$.
%Finally, let $P,Q\in\mathcal P$ and $T(P;x,y)=T(Q;x,y)$.
%Since $P\in\mathrm H$ and $Q\in\mathrm{SP}$,
%the main theorem, Theorem \ref{thm:main}, gives $P\cong Q$.
%Conversely, if $P\cong Q$, then the invariance of the Tutte polynomial
%under poset isomorphism gives $T(P;x,y)=T(Q;x,y)$.
%This proves the corollary.
\end{proof}

\begin{proof}[\rm\textbf{Proof of Theorem \ref{thm:main}}]
We first prove that, for every $P\in\mathrm H$ and every
$Q\in\mathrm{SP}$, the relation $P\cong Q$ implies $T(P;x,y)=T(Q;x,y)$.
Suppose that $P\cong Q$, and let $\varphi:P\to Q$ be a poset isomorphism.
For any $X\subseteq P$, write $\varphi(X)=\{\varphi(u):u\in X\}$.
By the definition of a poset isomorphism, for any $u,v\in P$,
the elements $u,v$ are comparable in $P$ if and only if
$\varphi(u),\varphi(v)$ are comparable in $Q$.
Thus, a subset $A\subseteq P$ is an antichain of $P$ if and only if
$\varphi(A)$ is an antichain of $Q$.
Since $\varphi$ is a bijection, the map $A\mapsto\varphi(A)$
is a bijection from the set of antichains of $P$ to the set
of antichains of $Q$, with inverse $B\mapsto\varphi^{-1}(B)$.

For any antichain $A\subseteq P$ and any $u\in P$,
the definitions of a generated order filter and a poset isomorphism
give the following equivalences:
$u\in I^*(A)$ if and only if there exists $v\in A$ such that
$v\le u$ in $P$, which holds if and only if there exists $v\in A$
such that $\varphi(v)\le\varphi(u)$ in $Q$, that is, $\varphi(u)\in I^*(\varphi(A))$.
Therefore, $\varphi(I^*(A))=I^*(\varphi(A))$, where $I^*(A)$ is generated in $P$ and
$I^*(\varphi(A))$ is generated in $Q$. Since $\varphi$ is a bijection, we have
$|\varphi(A)|=|A|$ and $|I^*(\varphi(A))|=|I^*(A)|$.
Hence, the terms corresponding to $A$ and $\varphi(A)$ in the formula \eqref{eq:antichain-expansion} are equal.
Summing over all antichains and using the bijection between the two sets of antichains yields $T(P;x,y)=T(Q;x,y)$.

We now prove the converse implication.
Suppose that $P\in\mathrm H$, $Q\in\mathrm{SP}$, and $T(P;x,y)=T(Q;x,y)$.
By the formula \eqref{eq:normalization-degree}, we obtain
\begin{eqnarray*}
|P|=\deg_x T(P;x,y)=\deg_x T(Q;x,y)=|Q|=n.
\end{eqnarray*}
We proceed by induction on $n$.
If $n=1$, then both $P$ and $Q$ are one-element posets,
and hence $P\cong Q$.
Now suppose that $n\geq2$, and assume that, for every
$R\in\mathrm H$ with $|R|<n$ and every $S\in\mathrm{SP}$,
the equality $T(R;x,y)=T(S;x,y)$ implies $R\cong S$.
Since $\mathrm H\subseteq\mathrm{SP}$,
Lemma \ref{lem:connectivity} and the equality
$T(P;x,y)=T(Q;x,y)$ imply that $G(P)$ and $G(Q)$
are either both connected or both disconnected.
We consider these two cases separately.

\textbf{Case 1:} $c(G(P))=c(G(Q))=1$.

By part \textup{(iii)} of Lemma \ref{lem:graph-decomposition},
order all connected components of $\overline{G(P)}$ and
$\overline{G(Q)}$ according to their respective partial orders as
$D_1,\ldots,D_r$ and $E_1,\ldots,E_s$, respectively, where $r,s\geq2$.
Let $P_i=P[V(D_i)]$ and $Q_j=Q[V(E_j)]$.
Then $P=P_1\oplus\cdots\oplus P_r$ and
$Q=Q_1\oplus\cdots\oplus Q_s$,
and all the factors are nonempty and ordinally indecomposable.
The equality $T(P;x,y)=T(Q;x,y)$ and  Lemma \ref{lem:ordinal-recovery} give $r=s$ and
\begin{eqnarray}\label{eq:main-ordinal-matching}
T(P_i;x,y)=T(Q_i;x,y)\quad(1\leq i\leq r).
\end{eqnarray}
Since $P\in\mathcal W(P)$ and $D_i\in\mathcal C(\overline{G(P)})$,
the definition of $\mathcal W(P)$ gives $V(D_i)\in\mathcal W(P)$.
Together with $P\in\mathrm H$, Lemma \ref{lem:H-inheritance} therefore yields $P_i=P[V(D_i)]\in\mathrm H$.
By \textup{(i)} of Lemma \ref{lem:graph-decomposition}, we have $Q_i\in\mathrm{SP}$.
Since $r\geq2$ and all the factors are nonempty, we have $1\leq|P_i|<n$.
Combining the formula \eqref{eq:main-ordinal-matching} with the induction hypothesis gives $P_i\cong Q_i$.

For each $1\leq i\leq r$, choose a poset isomorphism
$\varphi_i:P_i\to Q_i$, and define $\varphi:P\to Q$ by
$\varphi(u)=\varphi_i(u)$ for $u\in V(D_i)$.
Since $V(D_1),\ldots,V(D_r)$ are pairwise disjoint
and their union is $P$, for every $u\in P$
there is a unique index $i$ such that $u\in V(D_i)$.
Thus, $\varphi(u)=\varphi_i(u)$ is uniquely determined,
so $\varphi$ is well defined.
Moreover, each $\varphi_i$ maps $V(D_i)$ bijectively onto $V(E_i)$,
and $V(E_1),\ldots,V(E_r)$ are pairwise disjoint with union $Q$.
Therefore, $\varphi$ is a bijection.

Let $u,v\in P$ be distinct. If $u,v\in V(D_i)$, then, since $P_i$ and $Q_i$
are the corresponding induced subposets and $\varphi_i$
is a poset isomorphism, we have $(u,v)\in A(\vec G(P))$ if and only if
$(\varphi(u),\varphi(v))\in A(\vec G(Q))$.
If $u\in V(D_i)$ and $v\in V(D_j)$ with $i\neq j$,
then $\varphi(u)\in V(E_i)$ and $\varphi(v)\in V(E_j)$.
By the order of the factors in the two ordinal sum decompositions,
$(u,v)\in A(\vec G(P))$ if and only if $i<j$,
and $(\varphi(u),\varphi(v))\in A(\vec G(Q))$ if and only if $i<j$.
Thus, these two arc conditions are equivalent for all distinct $u,v\in P$.
Hence, $\varphi$ is a digraph isomorphism from $\vec G(P)$ to $\vec G(Q)$,
and consequently $P\cong Q$.

\textbf{Case 2:} $c(G(P))\geq2$ and $c(G(Q))\geq2$.

Let $\mathcal C(G(P))=\{D_1,\ldots,D_r\}$ and $\mathcal C(G(Q))=\{E_1,\ldots,E_s\}$, where $r,s\geq2$.
Set $P_i=P[V(D_i)]$ and $Q_j=Q[V(E_j)]$.
By part \textup{(ii)} of Lemma \ref{lem:graph-decomposition}, we obtain
\begin{eqnarray*}
P=P_1+\cdots+P_r,\qquad Q=Q_1+\cdots+Q_s,
\end{eqnarray*}
where all the factors are nonempty connected posets,
and $G(P_i)=D_i$, $G(Q_j)=E_j$.
Since $P\in\mathrm H$ and $c(G(P))=r\geq2$,
Lemma \ref{lem:H-inheritance} implies that each $P_i$ belongs to $\mathrm H$ and that $T(P_i;x,y)$
is irreducible in $\mathbb Q[x,y]$.
Moreover, since each $Q_j$ is a nonempty induced subposet of $Q$,
part \textup{(i)} of Lemma \ref{lem:graph-decomposition} gives $Q_j\in\mathrm{SP}$.

Repeatedly applying the direct sum formula in Lemma \ref{lem:sum-formulas}
and using $T(P;x,y)=T(Q;x,y)$, we obtain
\begin{equation}\label{eq:main-direct-product}
\prod_{i=1}^{r}T(P_i;x,y)=\prod_{j=1}^{s}T(Q_j;x,y).
\end{equation}
By the formula \eqref{eq:normalization-degree}, each $T(Q_j;x,y)$ is nonzero and nonconstant.
Furthermore, for each $D\in\mathcal C(G(P))$, the polynomial $T(P[V(D)];x,y)$ is irreducible in $\mathbb Q[x,y]$.
By the formula \eqref{eq:main-direct-product} and unique factorization in $\mathbb Q[x,y]$, there exist pairwise disjoint subsets
$\mathcal D_1,\ldots,\mathcal D_s$ of $\mathcal C(G(P))$ such that
\begin{eqnarray*}
\bigcup_{j=1}^{s}\mathcal D_j=\mathcal C(G(P)),
\end{eqnarray*}
and nonzero rational numbers $\lambda_1,\ldots,\lambda_s$ such that
\begin{equation}\label{eq:main-factor-allocation}
T(Q_j;x,y)=\lambda_j\prod_{D\in\mathcal D_j}T(P[V(D)];x,y)\quad(1\leq j\leq s).
\end{equation}
Each connected component of $G(P)$ indexes one occurrence
of an irreducible factor on the left-hand side of the formula \eqref{eq:main-direct-product}.
Distinct components with the same Tutte polynomial index distinct occurrences and are allocated separately.

If $\mathcal D_j=\varnothing$, then the above equality gives
$T(Q_j;x,y)=\lambda_j$, contradicting the fact that $T(Q_j;x,y)$ is nonconstant.
Thus, each $\mathcal D_j$ is nonempty.
Substituting $x=1$ into the formula \eqref{eq:main-factor-allocation}
and using the formula \eqref{eq:normalization-degree}, we have $T(Q_j;1,y)=1$ and $T(P[V(D)];1,y)=1$,
and hence $\lambda_j=1$. Therefore,
\begin{equation}\label{eq:main-factor-groups}
T(Q_j;x,y)=\prod_{D\in\mathcal D_j}T(P[V(D)];x,y)\quad(1\leq j\leq s).
\end{equation}
If some $\mathcal D_j$ contained at least two connected components,
then the above equality would express $T(Q_j;x,y)$ as a product of the Tutte polynomials
of at least two nonempty finite posets.
Since $Q_j\in\mathrm{SP}$ and $G(Q_j)=E_j$ is connected,
this would contradict Lemma \ref{lem:no-grouping}.
Consequently, for each $1\leq j\leq s$,
the set $\mathcal D_j$ contains exactly one connected component.
Since $\mathcal C(G(P))=\bigcup_{j=1}^{s}\mathcal D_j$, we obtain 
\begin{eqnarray*}
r=|\mathcal C(G(P))|=\sum_{j=1}^{s}|\mathcal D_j|=s.
\end{eqnarray*}
After suitably reindexing the connected components of $G(Q)$
and the corresponding direct sum factors simultaneously,
we may assume that $\mathcal D_i=\{D_i\}$.
Then by the formula \eqref{eq:main-factor-groups} and the above arguments, we have
\begin{equation}\label{eq:main-direct-matching}
T(P_i;x,y)=T(Q_i;x,y)\quad(1\leq i\leq r).
\end{equation}

Since $r\geq2$ and each $P_i$ is nonempty, we have $1\leq|P_i|<|P|=n$.
Combining $P_i\in\mathrm H$, $Q_i\in\mathrm{SP}$,
and the formula \eqref{eq:main-direct-matching}, the induction hypothesis yields $P_i\cong Q_i$.
For each $1\leq i\leq r$, choose a poset isomorphism
$\varphi_i:P_i\to Q_i$, and define $\varphi:P\to Q$ by
$\varphi(u)=\varphi_i(u)$ for $u\in V(D_i)$.
Since $V(D_1),\ldots,V(D_r)$ are pairwise disjoint
and their union is $P$, each $u\in P$ corresponds to a unique index $i$,
so $\varphi$ is well defined.
Moreover, each $\varphi_i$ maps $V(D_i)$ bijectively onto $V(E_i)$,
and $V(E_1),\ldots,V(E_r)$ are pairwise disjoint with union $Q$.
Thus, $\varphi$ is a bijection.

Let $u,v\in P$ be distinct.
If $u,v\in V(D_i)$, then the equalities
$\vec G(P_i)=\vec G(P)[V(D_i)]$ and $\vec G(Q_i)=\vec G(Q)[V(E_i)]$,
together with the fact that $\varphi_i$ is a poset isomorphism, give 
\begin{eqnarray*}
(u,v)\in A(\vec G(P))\quad\text{if and only if}\quad(\varphi(u),\varphi(v))\in A(\vec G(Q)).
\end{eqnarray*}
If $u\in V(D_i)$ and $v\in V(D_j)$ with $i\neq j$,
then $u$ and $v$ belong to different connected components of $G(P)$.
Hence, $\{u,v\}\notin E(G(P))$,
that is, they are incomparable in $P$.
Similarly, $\varphi(u)\in V(E_i)$ and $\varphi(v)\in V(E_j)$
belong to different connected components of $G(Q)$.
Thus, $\{\varphi(u),\varphi(v)\}\notin E(G(Q))$,
so they are also incomparable in $Q$.
Therefore, for all distinct $u,v\in P$, we have
\begin{eqnarray*}
(u,v)\in A(\vec G(P))\quad\text{if and only if}\quad(\varphi(u),\varphi(v))\in A(\vec G(Q)).
\end{eqnarray*}
Hence, $\varphi$ is a digraph isomorphism from $\vec G(P)$ to $\vec G(Q)$,
and consequently $P\cong Q$.

In both cases, we obtain $P\cong Q$, completing the induction.
Together with the implication established at the beginning,
this proves the theorem.
\end{proof}

\section{Conclusion}

This paper studies the problem of determining whether series-parallel posets are isomorphic from their Tutte polynomials.
Using the connected component decompositions of the comparability graph and its complement,
together with irreducibility conditions on the Tutte polynomials of the relevant induced subposets,
we define a subclass $\mathrm H$ of series-parallel posets and prove that, for every $P\in\mathrm H$ and every $Q\in\mathrm{SP}$,
the equality $T(P;x,y)=T(Q;x,y)$ holds if and only if $P\cong Q$.
We further prove that $\mathcal P\subseteq\mathrm H$, thereby recovering the result in \cite{Gordon1996}
on determining isomorphism within the class $\mathcal P$ as a special case of our main theorem.
These results provide a partial affirmative answer to Conjecture \ref{conj:reconstruction}
and show that the irreducibility of the Tutte polynomials of the relevant induced subposets plays an important role
in determining the isomorphism types of series-parallel posets.

\noindent{\bf Conflicts of Interest}

The authors declare that they have no conflicts of interest.

\noindent{\bf Data Availability}

No additional data are available.

\end{document}